\documentclass[12pt,reqno]{amsart}
\usepackage{amssymb}

\usepackage[curve,matrix,arrow]{xy}

\baselineskip

\theoremstyle{plain} 
\numberwithin{equation}{section}
\newtheorem{thm}[equation]{Theorem}
\newtheorem{lemma}[equation]{Lemma}

\newtheorem{prop}[equation]{Proposition}

\newtheorem{defi}[equation]{Definition}

\theoremstyle{definition}
\newtheorem{rem}[equation]{Remark}

\newtheorem{exm}[equation]{Example}

\theoremstyle{remark}

\def\CC{{\mathcal{C}}}

\def\CE{{\mathcal{E}}}
\def\CF{{\mathcal{F}}}
\def\CH{{\mathcal{H}}}

\def\CL{{\mathcal{L}}}
\def\CM{{\mathcal{M}}}
\def\CN{{\mathcal{N}}}

\def\CP{{\mathcal{P}}}

\def\CS{{\mathcal{S}}}

\def\CU{{\mathcal{U}}}
\def\CV{{\mathcal{V}}}

\def\CV{{\mathcal{V}}}

\def\FA{{\mathfrak{A}}}

\def\bZ{{\mathbb Z}}

\def\Id{\operatorname{Id}\nolimits}
\def\Proj{\operatorname{Proj}\nolimits}

\def\res{\operatorname{res}\nolimits}

\def\Rad{\operatorname{Rad}\nolimits}

\def\modg{\operatorname{{\bf mod}(\text{$kG$})}\nolimits}

\def\Modg{\operatorname{{\bf Mod}(\text{$kG$})}\nolimits}

\def\stmodg{\operatorname{{\bf stmod}(\text{$kG$})}\nolimits}
\def\Stmodg{\operatorname{{\bf StMod}(\text{$kG$})}\nolimits}
\def\Stmod{\operatorname{{\bf StMod}}\nolimits}

\def\sfg{\operatorname{{\mathcal S}{\mathcal F}(\text{$kG$})}\nolimits}

\def\HHH{\operatorname{H}\nolimits}
\def\Hom{\operatorname{Hom}\nolimits}
\def\PHom{\operatorname{PHom}\nolimits}

\def\End{\operatorname{End}\nolimits}
\def\Endul{\operatorname{\underline{End}}\nolimits}
\def\HH#1#2#3{\HHH^{#1}(#2,#3)}

\def\Homul{\operatorname{\underline{Hom}}\nolimits}

\def\Ext{\operatorname{Ext}\nolimits}

\def\hgs{\HH{*}{G}{k}}

\def\Ann{\operatorname{Ann}\nolimits}

\title[Idempotent Modules and local supports]
{Idempotent modules, locus of compactness and local supports}
\author[Jon F. Carlson]{Jon F. Carlson}
\address{Department of Mathematics, University of Georgia, 
Athens, GA 30602, USA}
\email{jfc@math.uga.edu}
\thanks{Research partially supported by 
Simons Foundation grant 054813-01}

\date\today

\dedicatory{To the memory of Brian Parshall and Georgia Benkart, good 
friends and wonderful colleagues}

\begin{document}

\begin{abstract} 
Let $kG$ be the group algebra of a finite group scheme defined over
a field $k$ of characteristic $p>0$. Associated to any closed subset $V$
of the projectivized prime ideal spectrum $\Proj \hgs$ is a thick 
tensor ideal subcategory of the stable category of finitely generated
$kG$-module, whose closure under arbitrary direct sums is a 
localizing tensor ideal in the stable category of all $kG$-modules. 
The colocalizing functor from the big stable category to this 
localizing subcategory is given by tensoring with an idempotent module
$\CE$. A property of the idempotent module is that its restriction 
along any flat map $\alpha:k[t]/(t^p) \to kG$ is a compact object. For
a $kG$-module $M$, we define its locus of compactness in terms 
of such restrictions. With some added hypothesis, in the case that 
$V$ is a closed point, for a $kG$-module $M$, we show that in the 
stable category $\Hom(\CE, M)$ is finitely generated over the endomorphism
ring of $\CE$, provided the restriction along an associated flat map is 
a compact object. This leads to a notion of local supports. We prove some
of its properties, give a realization theorem and display the images of 
the $L_\zeta$ modules under the colocalizing functor.  
\end{abstract}

\maketitle


\section{Introduction and notation}
Suppose that $k$ is a field of characteristic $p > 0$ and that $G$ is a
finite group scheme defined over $k$. Let $kG$ denote its group algebra, 
which is a finite dimensional cocommutative Hopf algebra. 
Examples of such algebras include the group algebras of finite groups,
restricted enveloping algebras of restricted $p$-Lie algebras and 
infinitesimal subgroups of algebraic groups defined over $k$.
For any such algebra the stable categories $\stmodg$ of finitely generated
$kG$-modules and $\Stmodg$ of all $kG$-modules are tensor triangulated
categories. There is a notion of support varieties, and the thick tensor 
ideals in $\stmodg$ have been classified using this construction and 
the spectrum $V_G(k) = \Proj \HHH^*(G,k)$ of the cohomology ring
\cite{BCR, FP}.

Associated to any thick tensor ideal in $\stmodg$ 
is a triangle consisting of the trivial module and a 
pair of  idempotent modules $\CE$ and $\CF$. Tensoring with these
modules induce colocalizing and localizing functors on the stable 
category. Except in trivial cases, these module are infinitely 
generated. None the less, they have the unusual property that 
their restrictions along any flat map (called a $\pi$-point) 
$A = k[t]/(t^p) \to kG$ is
a compact object in $\Stmod(A)$. So it would seem that the property 
of having compact restrictions is related to the structure of the 
category $\Stmodg$.  In this paper we begin  an investigation of this issue.  

We introduce a locus of compactness for a $kG$-module $M$. 
It is the collection of all 
points in $V_G(k)$ such that the restriction 
of $M$ along any associated $\pi$-point 
is compact. The locus has some properties similar to support varieties,
and in particular, it identifies some thick subcategories of 
$\Stmodg$. In the case that $kG$ is the restricted enveloping algebra
of a Lie algebra, the properties are particularly nice, and 
the subcategories closed under tensor products. In other cases 
there is some dependence on the Hopf structure on $kG$. Given a module $M$, 
the inclusion of a point of $V_G(k)$ in the locus of $M$ depends on the 
image of $M$ in the colocalized category generated by the idempotent
$\CE$ module. 

The idempotent $\CE$ module acts by tensor product as the identity on the 
colocalized subcategory that it generates. As a consequence, the group of 
homomorphisms between two modules in the subcategory is module over the 
endomorphism ring of the idempotent module. Using earlier work \cite{Cendo1},
we know precisely the structure of the endomorphism ring of the $\CE$ module 
in the case that $kG$ is the group algebra of an elementary abelian $p$-group
and the variety of the $\CE$ module is a single point $V$ in $V_G(k)$. While it
is true in this case that the endomorphism ring is 
infinitely generated, modules 
in the subcategory can still be distinguished by the annihilator of their 
endomorphisms and homomorphisms. 
We show that, given any noninveritble element $\zeta$ of the 
stable endomorphism ring $\Endul_{kG}(\CE)$,
there is a $kG$-module $M_\zeta$ such that the annihilator of 
$\Homul_{kG}(\CE, M)$ in $\Endul_{kG}(\CE)$
is close to being the ideal generated by $\zeta$. 
That is, the radicals of the ideals coincide. 
In addition, if the point $V$ is in the locus of 
compactness of a module $M$ in the subcategory, then the homomorphism 
group in the stable category $\Homul_{kG}(\CE, M)$ is finitely generated 
over the endomorphism ring $\Endul_{kG}(\CE)$. 

After an early version of this paper was poseted, it came to light that
the group consisting of Dave Benson, Srikanth Iyengar, Henning Krause and
Julia Pevtsova have results that overlap significantly with those in this 
paper. The BIKP collaboration approaches the subject from an abstract 
categorical perspective that contrasts sharply with the more elementary 
hands-on development presented here. In several cases, they prove stronger 
theorems, but without as much of 
the detailed information on the modules and maps. 
I am grateful to the group for letting me see their 
unfinished manuscript \cite{BIKP2} and
letting me borrow at least one of their ideas. I want to also thank
the Hausdorff Research Institute for Mathematics at the University of 
Bonn for the support and stimulating program during which some of 
work on this paper was completed. 

For references on group representation and cohomology see the 
texts \cite{Bbook} or \cite{CTVZ}. For triangulated categories,
see \cite{Neem} or the early sections of \cite{BF}.


\section{Background}
Throughout the paper $kG$ is the group algebra of a finite group 
scheme $G$ that is defined over the field $k$ of characteristic $p > 0$.
Let $\modg$ denote
the category of finitely generated $kG$-modules and 
$kG$-module homomorphisms, and let $\Modg$ be the category 
of all $kG$-modules and homomorphisms.  The stable 
category $\stmodg$ has the same objects as $\modg$, but in the stable 
category the set of morphisms from object $M$ to object $N$
is given by 
\[
\Homul_{kG}(M,N) = \frac{\Hom_{kG}(M,N)}{\PHom_{kG}(M,N)}
\]
where $\PHom_{kG}(M,N)$ is the set of all homomorphisms that factor 
through a projective module. The category $\Stmodg$ of all $kG$-modules
is constructed the same way, by factoring out any map that factors through
a projective module. Note that $\stmodg$ is
the subcategory of compact objects in $\Stmodg$.

For a $kG$-module $M$, the modules $\Omega(M)$ and $\Omega^{-1}(M)$ are 
defined to be the kernel of a projective 
cover $P \twoheadrightarrow M$ and the 
cokernel of an injective hull $M \hookrightarrow I$, respectively. Recall that 
$kG$ is a self-injective rings, so that the injective module $I$ is 
projective and the projective module $P$ is injective. Consequently, 
$\Omega(\Omega^{-1}(M)t) \cong \Omega^{-1}(\Omega(M)$ is the largest 
direct summand of $M$ that has no projective summands.

The stable categories $\stmodg$ and $\Stmodg$ are tensor triangulated
categories. The tensor product of modules is over the base field $k$ with
the action of $kG$ defined using the coalgebra structure on $kG$. The 
triangles correspond roughly to short exact sequences in the module 
category. That is, a triangle has the form 
\[
\xymatrix{
{} \ar[r] & X \ar[r]^\alpha & Y \ar[r]^\beta & Z \ar[r]^{\gamma \quad} & 
\Omega^{-1}(X) \ar[r] &  {}
} 
\]
where for some projective module $P$ and maps $\alpha^\prime$, $\beta^\prime$
representing the classes $\alpha$ and $\beta$, there is an exact sequence 
\[
\xymatrix{
0 \ar[r] & X \ar[r]^{\alpha^\prime \quad}  & Y \oplus P 
\ar[r]^{\quad \beta^\prime} & 
Z \ar[r] & 0 \quad .
}
\]
The operator $\Omega^{-1}$ is the 
translation or shift functor on the stable categories.  

A $\pi$-point \cite{FP} is a flat map 
$\alpha: K[t]/(t^p) \to KG_K$, where $K$ is 
some extension of $k$, and the map factors by flat maps through 
the group algebra $KE$ of some unipotent abelian subgroup 
scheme $E$ of $G_K$. Two $\pi$-points,
$\alpha_K: K[t]/(t^p) \to KG_K$ and $\beta_L: L[t]/(t^p) \to LG_L$, 
are equivalent if for any finitely generated $kG$-module $M$ the restriction
of $K \otimes_k M$ to $K[t]/(t^p)$, denoted 
$\alpha_K^*(K \otimes M)$, and $\beta_L^*(L \otimes M)$ are either both 
projective or both not projective. 

In the case that $G$ is a finite group, a unipotent subgroup scheme would be 
a subalgebra $kE$ where $E = \langle g_1, \dots, g_r \rangle$ is an 
elementary abelian subgroup of order $p^r$ for some $r$. Then a $\pi$-point
that is defined over $k$ would have the form 
$\alpha:k[t]/(t^p) \to kE \subseteq kG$ given by 
$\alpha(1) = \sum_{i=1}^r \alpha_i(g_i-1) + w$ where $w$ is some element in 
$\Rad^2(kE)$ and for some $i$, $\alpha_i \neq 0$.  
The equivalence class of $\alpha$ depends only on the 
element $[\alpha_1| \dots| \alpha_r] \in \CP_k^{r-1}$. Let $V^r_G(k)$
denote the collection of all equivalence classes of $\pi$-points. 

The cohomology ring $\hgs$ is a finitely generated graded-commutative 
$k$-algebra. This was proved by Evens, Golod and Venkov for finite 
groups and by Friedlander and Suslin \cite{FSus} 
for general finite group schemes. 
The projectivized prime ideal 
spectrum $V_G(k) = \Proj \hgs$ is a 
finite dimensional projective variety. 
For $M$ a finitely generated $kG$-module, let $V_G(M) \subset V_G(k)$ 
be the collection of all homogeneous prime ideals that contain the 
annihilator of $\Ext^*_{kG}(M,M)$.

Given a closed subvariety $V \subseteq V_G(k)$, the collection $\CM_V$ of all
finitely generated $kG$-modules $M$ with $V_G(M) \subseteq V$ is a 
thick tensor ideal in $\stmodg$, meaning that it is a triangulated 
subcategory that is both closed under taking direct summands and 
closed under tensor products with arbitrary objects in $\stmodg$. 

For any such $V$ and $\CM_V$, there is a distinguished 
triangle \cite{R} in $\Stmodg$ having the form 
\[
\xymatrix{
\CS_V: & 
{} \ar[r] & \CE_V \ar[r]^{\sigma_V} & k \ar[r]^{\tau_V} & \CF_V \ar[r] & 
\Omega(\CE_V) \ar[r] & {}
}
\]
where $\CE_V$ and $\CF_V$ are idempotent modules in that 
$\CE_V \otimes \CE_V \cong \CE_V$, and $\CF_V \otimes \CF_V \cong \CF_V$
in the stable category. In addition $\CE_V \otimes \CF_V \cong 0$, 
meaning that $\CE_V \otimes \CF_V$ is a projective module. 
The triangle has a couple of universal properties. Let $M$ be a 
finitely generated $kG$-module. The first property is that any map
$\gamma: N \to M$, with $N$ in $\CM_V$, factors through 
$\sigma_V \otimes \Id_M: \CE_V\otimes M \to k \otimes M \cong M$. 
The other property is that any map $\gamma: M \to N$, 
such that the third object in the triangle
of $\gamma$ is in $\CM_V$, factors through $\tau_V \otimes \Id_M:
M \cong k \otimes M \to \CF_V \otimes M$. 

As to support varieties, for any $kG$-module $M$, $\CV_G(M)$ is the 
collection of all points in $V_G(k)$ with the property that the 
pull back along a corresponding $\pi$-point $\alpha_K: K[t]/(t^p)
\to KG_K$ of the extended module $K \otimes M$ is not projective. 
In the case of the idempotent modules, 
with $V$ closed, $\CV_G(\CE_V)$ is the set 
of all points in $V$, while 
$\CV_G(\CF_V) = \CV_G(k) \setminus \CV_G(\CE_V)$, its complement.


\section{The locus of compactness}
First we make a series of observations that lead to the 
definition of many thick subcategories of $\Stmodg$ associated to 
subsets the projectived prime ideal spectrum of the cohomology 
ring $\HHH^*(G,k) \cong \Ext^*_{kG}(k,k)$. We show that under certain 
circumstance, there exist some different sorts of support varieties
based on cardinalities of dimensions of restrictions of the objects along 
$\pi$-points.  

By an idempotent module, we always mean a $kG$-module $M$ such 
that $M \otimes M \cong M$ in the stable category. Or said another 
way, $M$ is an idempotent module if $M \otimes M \cong M \oplus P$
where $P$ is a projective module. We recall that the only finitely 
generated idempotent module is the trivial module $k$. 
We refer to the modules $\CE_V$ and $\CF_V$, for $V$ closed
in $V_G(k)$ as the Rickard idempotent modules. 
In addition, there are many other idempotent modules, including those
defined for more general collections of points in $V_G(k)$. 
For example, a countable 
direct sum of copies of the trivial module is an idempotent module. 
If we let $M = \sum_{n\geq 0} \Omega^n(k)$, then a countable direct
sum of copies of $M$ is an idempotent module. 

For all of this, the Rickard idempotent modules have a very special 
property. To be precise we have the following. 

\begin{lemma} \label{lem:ric}
Suppose that $M$ is one of $\CE_V$ or $\CF_V$ for $V$ a closed subvariety
of $V_G(k)$. Then for any $\pi$-point $\alpha_K: K[t]/(t^p) \to KG_K$
we have that the restriction $\alpha_K(K \otimes M)$ is either the 
zero module or is isomorphic to $K$ in $\Stmod(K[t]/(t^p))$. 
The same is true of 
the tensor $\CE_V \otimes \CF_{V^\prime}$ for $V$ and $V^{\prime}$ 
closed subvarieties. 
\end{lemma}

\begin{proof}
Consider the extended triangle $K \otimes_k \CS_V$. If the $\pi$-point 
$\alpha_K$ corresponds to a point in $V$, then it is does not correspond
to any point in $\CV_k(\CF_V)$. Consequently, $\alpha_K^*(K \otimes \CF_V)$ 
is a projective module, and by the triangle, 
$\alpha_K^*(K \otimes \CE_V) \cong K$ in the stable category. 
On the other hand,  if $\alpha_K$ 
does not correspond to a point in $V$, then $\alpha_K^*(K \otimes \CE_V)$
is projective and $\alpha_K^*(K \otimes \CF_V) \cong K$.
\end{proof}

The above lemma implies that all of the Rickard idempotent modules are 
contained in the full subcategory $\sfg$ of $\Stmodg$ consisting of all 
$kG$-modules $M$  with the property that $\alpha^*_K(K \otimes M)$ is finite 
dimensional for all $\pi$-points 
\[
\alpha_K: K[t]/(t^p) \to KG_K. 
\]
In other words, the restriction along any $\pi$-point of a module in 
$\sfg$ is a compact object in the stable category of $K[t]/(t^p)$.
The category $\sfg$ is triangulated because the restriction maps 
preserve triangles and if two objects in a triangle are compact then
so it the third. It is easy to see that it is also closed under 
finite direct sums and taking direct summands. 
In addition, by an arguments that is very similar to that of the proof
of Lemma \ref{lem:ric}, it contains the tensor product of any 
finite dimensional module with a Rickard idempotent module. 

An observation of Benson, Iyengar, Krause and Pevtsova helps us to make 
sense of this. 

\begin{prop} \label{prop:equi-pi} \cite{BIKP2}
Suppose that $\alpha, \beta: k[t]/(t^p) \to kG$ 
are equivalent $\pi$-points that are defined over $k$. 
Let $M$ be a $kG$-module such that $\alpha^*(M)$
is a compact object in $\Stmod(k[t]/(t^p))$. Then  
$\beta^*(M)$ is also compact.
\end{prop}

\begin{proof}
Let $M^* = \Hom_k(M, k)$ be the $k$-dual of $M$.  
There is a natural map $M \to M^{**}$ from $M$ to its double dual.
This is a $kG$-homomorphism. That is, while the action of $kG$ on $M^*$
is defined using the antipode of the Hopf structure, the action on
$M^{**}$ involves applying the antipode twice, which is the identity. 
Assume  $\alpha^*(M)$ is compact. Then the map to the double dual is an 
isomorphism in the stable category.  In particular,  
the cokernel $U$ of the natural map $M \to M^{**}$, which is an injective map, 
is a projective module on restriction along $\alpha$. But then, $\beta^*(U)$
is also projective since $\beta$ is equivalent to $\alpha$. 
It follows, that $\beta^*(M)$ is a compact object. 
\end{proof}

Suppose that $\CV$ is any subset of points in $V_G(k)$.
Let $\CN_{\CV}$ be the full subcategory of all $kG$-modules $M$
with the property that for any $\pi$-point $\alpha_K: K[t]/(t^p) \to KG_K$
whose equivalence class is
in the collection $\CV$, the restriction $\alpha^*_K(K \otimes M)$ is
a compact object, meaning isomorphic to a
finite dimensional object in the stable category.
Then $\CN_{\CV}$ is a thick triangulated subcategory
of $\Stmodg$.
All of this suggests a variant support variety. We thank Paul Balmer 
for suggesting the name.

\begin{defi}  \label{def:newsup}
For a $kG$-module $M$, let $\CU_G(M)$ be the collection of points in
$V_G(k)$ with the property that for any $\pi$-point
associated to this point,
\[
\xymatrix{
\alpha_K: A = K[t]/(t^p) \ar[r] & KG_K,
} 
\]
the restriction $\alpha^*_K(K \otimes M)$ is
isomorphic to a finite dimensional module in the stable category
$\Stmod(A)$. We call $\CU_G(M)$ the locus of compactness or compact locus
of the module $M$. 
\end{defi}

Note that $\CU_G(M)$ is not a closed subvariety of $V_G(k)$, but rather
only a subset of points. Indeed, we see from item (2) in the next 
proposition, that $\CU_G(M)$ is likely to contain an open set of 
$V_G(M)$. We list some of the properties of $\CU_G$.

\begin{prop} \label{prop:propU}
Suppose that $M$ and $N$ are $kG$-modules.
\begin{enumerate} 
\item If $M$ is a compact object or if $M$ is in $\sfg$, then 
$\CU_G(M) = V_G(k)$.
\item If $V$ is a point that is not in  $\CV_G(M)$, then $V \in \CU_G(M)$.
\item As noted above, any subset $\CV \subset V_G(k)$ defines a 
thick subcategory $\CN_{\CV}$ in $\Stmodg$. 
\item Suppose that $L \to M \to N$ is a triangle in $\Stmod(kG)$. Then
\[
\CU_G(L) \cap \CU_G(N) \subset \CU_G(M).
\]
\item Suppose that $V$ is a closed subvariety of $V_G(k)$. Then 
\[
\CU_G(M \otimes \CE_V) = (V \cap \CU_G(M)) \cup c(V) \quad \text{ and }
\]
\[
\CU_G(M \otimes \CF_V) = V \cap (\CU_G(M) \cap c(V)).
\]
where $c(V)$ is the complement of $V$ in $V_G(k)$. 
In particular, if $M$ is a compact object, 
then $\CU_G(M \otimes \CE_V) = V_G(k)
= \CU_G(M \otimes \CF_V)$.
\end{enumerate}
\end{prop}
\begin{proof}
All but the last of the items are obvious. For the last item, suppose that
$\alpha_K:K[t]/(t^p) \to KG_K$ is a $\pi$-point for some extension $K$ of 
$k$. If the equivalence class of $\alpha_K$ is in $V$, then 
$\alpha_K^*(K \otimes \CF_V)$ is projective. It follows that 
in this case, 
\[
\alpha_K^*(K \otimes (\CE_V \otimes M)) = 
\alpha_K^*(K \otimes M). 
\]
Hence, $V \cap \CU_G(M) = V \cap \CU_G(\CE_V \otimes M).$

On the other hand, suppose that the class of $\alpha_K$ is not 
in $V$. Then $\alpha_K^*(K \otimes \CE_V)$ is projective and so is
$\alpha_K^*(K \otimes (\CE_V \otimes M))$. Hence, the class
of $\alpha_K$ is in $\CU_G(M)$. This proves the first statement.
The proof of the decomposition of $\CU_G(M \otimes \CF_V)$
is very similar. The last statement follows from the previous ones,
noting that $\CU_G(M) = V_G(k)$. 
\end{proof}

We remark that statement (4) of the proposition turns the usual 
condition for support varieties on its head. That is, for 
$L \to M \to N$ a triangle in $\Stmodg$ we have that 
\[
\CV_G(M) \quad \subseteq \quad \CV_G(L) \cup \CV_G(N)
\]
which is something like a reverse of the condition in (4).

There is one case where we can go even further with known results.

\begin{exm} \label{ex:restr-envel}
Suppose that $kG$ is the restricted enveloping algebra of a
restricted $p$-Lie algebra. In this case, $V_G(k)$ is identified with the 
restricted null cone $\CN_1$ (see \cite{FPar, Jan}). That is, we have an
embedding of the Lie algebra into its restricted enveloping algebra, which is 
the free tensor algebra on the Lie algebra modulo the relations given by the 
Lie bracket and the $p^{th}$-power operation.
The restricted null cone $\CN$ is the set of all elements $x$ of
the Lie algebra such that $x^{[p]} = 0$, where $x \mapsto x^{[p]}$ is the 
$p^{th}$-power operation.
In this case, every element of $V_G(k) \cong \CN$ has a unique
distinguished $\pi$-point that is a map of Hopf algebras. Consequently,
we can define $\CU_G(M)$ to be the collection of all
points in $V_G(k)$ such that the restriction 
of $M$ along the distinguished $\pi$-point associated to this point
is compact. That is, instead of letting $\CU_G(M)$, for a module
$M$, be a set of equivalence classes of $\pi$-points with the 
finite dimensionality property
for {\it all} elements of the equivalence class, we measure $\CU_G(M)$
only on the distinguished $\pi$-point in each class. 

The main fact is that if $\alpha_K:K[t]/(t^p) \to KG_K$ is such a 
distinguished $\pi$-point, then $\alpha_K$ is a Hopf algebra map. Most
importantly, restriction along $\alpha_K$ commutes with the tensor
product operation. This means that there is a tensor product theorem,
which  says for modules $M$ and $N$ that 
\[
\CU_G(M \otimes N) = (\CU_G(M) \cap \CU_G(N)) \cup c(\CV_G(M)) 
\cup c(\CV_G(N)),
\]
where $c(\CV_G(M))$ is the complement of $\CV_G(M)$ in $\CV_G(k)$.
That is, for a distinguished $\pi$-point $\alpha_K$, 
$\alpha^*_K(K \otimes (M \otimes N))$ is finite dimensional in the stable
category if the restrictions of both $M$ and $N$ 
are finite dimensional or if one of the 
two is projective. In particular, if $V$ is closed subvariety of 
$V_G(k)$ we have that $\CN_V$ is closed under tensor products in $\Stmodg$. 
\end{exm}


\section{Dependence on Hopf structure} \label{sec:hopf}
We notice that in some cases a given algebra over $k$ may have many possible
Hopf algebra structures. In particular, if $B$ is the truncated polynomial
ring $B = k[t_1, \dots, t_r]/(t_1^p, \dots, t_r^p)$, then we can make $B$ into 
a Hopf algebra by choosing any collection $X_1, \dots X_r \in \Rad(kG)$ 
that generate the radical of $B$. This means that the elements should be 
linearly independent modulo $\Rad^2(B)$. Then the elements $1+X_1, 
\dots, 1+X_r$ generate an elementary abelian subgroup $H$ order $p^r$ in the 
group of units of $B$, and $B$ can be taken to be the group algebra $kH$. We
can then impose the Hopf algebra structure of $kH$ on $B$. Similarly we 
can take the subspace of $B$ generated by $X_1, \dots, X_r$ to be a 
commutative restricted Lie algebra $L$ with trivial $p^{th}$-power operation. Then
$B$ is isomorphic to the restricted enveloping algebra of $L$, and it can
be equipt with the induced Hopf structure. 

One implication of the above is the following, a result that we find useful.

\begin{lemma} \label{lem:extendHopf}
Let $kG = k[t_1, \dots, t_r]/(t_1^p, \dots, t_r^p)$.
Assume that $A = k[t]/(t^p)$ is made into a Hopf algebra by regarding it 
either as a group algebra (the group being generated by $g = 1+t$) or as a
restricted enveloping algebra of the one dimensional Lie algebra $<t>$.
Then for any $\pi$-point defined over $k$, $\alpha: A \to kG$, there is a 
Hopf algebra structure on $kG$ such that $\alpha$ is a homomorphism of 
Hopf algebras.
\end{lemma}

The ability to choose the Hopf algebra structure is important 
when considering restrictions on the categories. In particular, if we 
are given a subalgebra $A$ of $kG$ that is a flat embedding, then the 
restriction map to the subalgebra induces a functor on stable categories. 
However, it does not preserve the tensor structure unless $A$ is a 
Hopf subalgebra. 

On the other hand, there are some cases where it does 
not matter. The following is one such case. 

\begin{prop} \label{prop:tensor}
Suppose that $kG$ is a finite group scheme and 
$\alpha: A = k[t]/(t^p) \to kG$ is a $\pi$-point defined over $k$.
Let $V$ be the closed point of $V_G(k)$ associated to $\alpha$ and let
$\CE_V$ be the associated idempotent module. For any
$kG$-module $M$, the restriction $\alpha^*(M)$ is isomorphic in
$\Stmod(A)$ to a finite dimensional module if and only if the restriction
$\alpha^*(\CE_V \otimes M)$ is also isomorphic to a finite dimensional
module.
\end{prop}

\begin{proof}
The point is that the subcategory $\CM_V$ of all $kG$-modules whose variety
is contained in $V$ is the same regardless of the tensor product. Likewise,
the triangle 
\[
\xymatrix{
\CS_V: &
{} \ar[r] & M \otimes \CE_V \ar[r]^{1 \otimes \sigma_V} & 
M \otimes k \ar[r]^{1 \otimes \tau_V} & M \otimes \CF_V \ar[r] & {}
}
\]
is distinguished by universal properties that don't depend on the 
tensor product (see Theorem 2.6 of \cite{BF}). 
In particular, $\alpha^*(M \otimes \CF_V) = \{0\}$ in the stable 
category $\Stmod(A)$, since the variety of $\CF_V$ does not 
contain $V$. Hence, we have an isomorphism 
$\alpha^*(M \otimes \CE_V) \cong \alpha^*(M)$. 
\end{proof}


\section{Endomorphism rings}  \label{sec:endomorph}
In this section we investigate the endomorphism rings of the module $\CE_V$ and
associated modules. Earlier work by Daugulis \cite{Dau}, showed that
under reasonable hypotheses on $V$, $\Endul_{kG}(\CE_V)$ is local and 
locally nilpotent. Krause \cite{Kra} derived some properties of $\CE_V$
using the observation that $\CE_V$ is endofinite. He also showed that 
$\Endul_{kG}(\CE_V)$ is commutative, since it is the center of the 
localizing subcategory of all $kG$-modules with variety contained in $V$.
That is, it is the algebra of natural transformation on the identity 
functor on that category. In this section, we give an explicit 
calculation of $\Endul_{kG}(\CE_V)$ in a specific case. 

We assume throughout that 
$kG = k[t_1, \dots, t_r]/(t_1^p, \dots, t_r^p)$ is the group algebra of 
an elementary abelian $p$-group or the restricted enveloping algebra 
of a  commutative Lie algebra with vanishing $p^{th}$-power operation. 
We assume a coalgebra structure, but in 
the end, it does not matter what that structure is. 

Assume that we have a $\pi$-point $\alpha: k[t]/(t^p) \to kG$ representing
a closed point $V$ in $V_G(k)$. The structure of the idempotent module
$\CE_V$ is presented in  \cite{Cendo1}. It is described as follows. For 
notation, let $kC$ be the image of $\alpha$ and let $Z = \alpha(t)$. 
Choose elements $X_1, \dots, X_{r-1}$ in $\Rad(kG)$ such 
$X_1, \dots, X_{r-1}, Z$ generate $\Rad(kG)$. Let $kH$ be the subalgebra
generated by $1$ and $X_1, \dots, X_{r-1}$, so that $kG \cong kH \otimes kC$.

\begin{thm} \label{thm:strEV}  \cite[Proposition 5.4]{Cendo1}
Suppose that 
\[
\xymatrix{
(P_*, \varepsilon): & {} \ar[r] & P_2 \ar[r]^{\partial} &
P_1 \ar[r]^{\partial}  & P_0 \ar[r]^{\varepsilon} & k \ar[r] & 0
}
\]
is a minimal projective $kH$-resolution of $kG$. 
Let $U$ be the inflation to $kG$ of the 
indecomposable $kC$-module of dimension $p-1$, so that 
$U_{\downarrow C} \cong \Omega_{kC}(k)$, has basis $1, Z, \dots, Z^{p-2}$
and has trivial action by $H$. Note that if $p=2$ then $U \cong k$. 
Then the restriction of $\CE_V$ to $kH$ is the direct sum 
\[
(\CE_V)_{\downarrow H} \cong P_0 \oplus (P_1 \otimes U) \oplus P_2 \oplus 
(P_3 \otimes U) \oplus \dots .
\]
The action of $Z$ on $\CE_V$ is given by the following formula. 
Assume that $x \in P_n$. For $n$ odd let 
\[
Z(x \otimes Z^j) = \begin{cases} x \otimes Z^{j+1} \in P_n \otimes U
& \text{ if } 0 \leq j < p-2, \\
\partial(x) \in P_{n-1} & \text{ if } j = p-2,
\end{cases}
\]
while for $n$ even
\[
Z(x) = \begin{cases} \partial(x) \otimes 1 \in P_{n-1} \otimes U 
& \text{ if } n > 0, \\
0 & \text{ if } n = 0.   \end{cases}
\]
\end{thm}

Next we consider the endomorphism ring of $\CE_V$. The fact that 
$(P_*, \varepsilon)$ is a minimal resolution implies that 
$\partial(P_i) \subseteq \Rad_{kH}(P_{i-1})$ for all $i \geq 1$. 
As a consequence, $\Homul_{kG}(\CE_V, k) \cong \prod_{i=0}^\infty 
\Hom_{kH}(P_i,k)$. That is, the minimality of the resolution also 
means that $\Hom_{kH}(P_i, k) \cong \HHH^i(H,k)$.

In addition, applying $\Homul_{kG}(\CE_V, - )$ to the distinguished triangle 
$\CS_V$, we get that 
\[
\xymatrix@+3pc{
\Homul_{kG}(\CE_V, \CE_V) \ar[r]^{(\sigma_V)_*} & \Homul_{kG}(\CE_V,k)
}
\]
is an isomorphism. Putting these facts together, 
we have established the following. 

\begin{prop} \label{prop:endoE1}
The endomorphism ring $\Endul_{kG}(\CE_V) \cong \prod_{i=0}^\infty
\HHH^i(H,k).$ as an additive group. 
\end{prop}

Any element $\gamma \in \HHH^d(H, k)$ is represented by a unique 
cocycle $\gamma: P_d \to k$, which in turn, extends to a chain map 
$\{\gamma_i\}: P_* \to P_*$ that is unique up to homotopy. Here 
$\gamma_i$ maps $P_{d+i}$ to $P_i$. The product on the cohomology 
ring $\HHH^*(H,k)$ can be taken to be the homotopy class of
the composition of these chain maps. Moreover, as in the proof of 
\cite[Theorem 7.1]{Cendo1} such a chain map induces an endomorphism
$\hat{\gamma}: \CE_V \to \CE_V$. 

In the case that the degree $d$ of $\gamma$ is even
or that $p=2$, the map $\hat{\gamma}$ is 
given on the $kH$-summands as follows. For $n$ odd, and $x \in P_{n+d}$, 
\[
\hat{\gamma}_n(x \otimes Z^j) = \gamma_n(x) \otimes Z^j,
\]
while for $n$ even, $\hat{\gamma}_n: P_{n+d} \to P_n$ by
\[
\hat{\gamma_n}(x) = \gamma_n(x).
\]
The formula is somewhat more complicated when $d$ is an odd 
integer and $p>2$. In that case,  for $n$ odd, 
$\hat{\gamma}_n: P_{n+d} \to P_n \otimes U$ is given by  
$\hat{\gamma_n}(x) = \gamma_n(x) \otimes Z^{p-2}$. For $n$ even, 
define $\hat{\gamma}_n: P_{n+d} \to P_n \oplus (P_{n-1} \otimes U)$ 
(assuming that $P_{-1} = \{0\}$) by 
\[
\hat{\gamma}_n(x \otimes Z^j) = \begin{cases} 
(\gamma_n(x), 0) & \text{ if } j = 0 \\
(0, \partial\gamma_n(x) \otimes Z^{j-1}) & \text{ if } j > 0
\end{cases}
\]
Notice that in all cases, $\hat{\gamma}$ is a $kH$-homomorphism since $\CE_V$
is free as a $kH$-module. To prove that it is a $kG$-homomorphism, it is 
only necessary to show that $\hat{\gamma}$ commutes with the action of $Z$.
We leave this as an exercise for the reader. 

Taking compositions, we get the following.

\begin{lemma}  \label{ref:endoprods} 
Suppose that the characteristic $p$ is odd. 
Let $\hat{}: \HHH^*(H,k) \to \Endul_{kG}(\CE_V)$ be the map that sends the 
cohomology element to the endomorphism as described above. Then we have
for $\gamma \in \HHH^n(G,k)$ and $\mu \in \HHH^m(G,k)$, the product is given 
as follows. 
\begin{enumerate}
\item If either $n$ or $m$ is even, the $\hat{\gamma}\hat{\mu} = 
\widehat{\gamma\mu}$.
\item if both $n$ and $m$ are odd, then $\hat{\gamma}\hat{\mu} = 0$.
\end{enumerate}
\end{lemma}

\begin{proof}
The proof of the first item is a matter of choosing chain maps representing 
the two elements and using the formula for the composition as above. 
For the second, we note that the 
image of the composition of the two chain maps is 
contained in $Z^{p-2} \CE_V$. But this is in the kernel of the map 
$\sigma_V: \CE_V \to k$. Because $\sigma_V^*: \Homul_{kG}(\CE_V, \CE_V) \to 
\Homul_{kG}(\CE_V, k)$ is an isomorphism we conclude that
$\hat{\gamma}\hat{\mu} = 0$.
\end{proof}

In this way we can characterize $\Endul_{kG}(\CE_V)$. 
We must first define a variant of the cohomology ring. 
Assume that $p >2$. Recall that 
\[
\HHH^*(H,k) \cong k[\zeta_1, \dots, \zeta_{r-1}] \otimes 
\Lambda(\eta_1, \dots, \eta_{r-1})
\]
where $\Lambda$ is the exterior algebra generated by the degree one 
elements $\eta_i$ and the elements $\zeta_i$ have degree 2. Let 
$R \subset \Lambda$ be the subalgebra spanned by the elements of even degree. 
Let $\Theta = R[\eta_1^\prime, \dots, \eta_{r-1}^\prime]/J$
where $J$ is the ideal generated by all products 
$\eta_i^\prime \eta_j^\prime$ and all monomials $\eta_i^\prime x$ 
such that $x \in R$ and $\eta_i x =0$ in $\Lambda$. Notice that 
$\Theta$ is a graded $k$-algebra and as $k$-vector spaces, 
$\Theta \cong \Lambda$. Let $\Gamma = 
k[\zeta_1, \dots, \zeta_{r-1}] \otimes \Theta$. Note that $\Gamma$ 
is a commutative $k$-algebra, and that we have an isomorphism
of $k$-vector spaces $\varphi: \HHH^*(G,k) \to \Gamma$ that is 
the identity on $k[\zeta_1, \dots, \zeta_{r-1}]$ and on $R$ 
and sends $\eta_i$ to $\eta_i^\prime$. 

In the case that $p=2$, let $\Gamma = \HHH^*(G,k)$ and let $\varphi$
be the identity. Then we have the following. 

\begin{prop} \label{prop:endoE2}
If $\alpha = (\alpha_1, \alpha_2, \dots)$ and 
$\beta = (\beta_1, \beta_2, \dots)$ are elements of $\Endul_{kG}(\CE_V)$
as in Proposition \ref{prop:endoE1} with $\alpha_i, \beta_i \in 
\HHH^i(H,k)$, then the product $\alpha\beta = \gamma = (\gamma_1,
\gamma_2, \dots)$ is given by
the relation 
\[
\gamma_n = \sum_{i=0}^n 
\varphi^{-1}(\varphi(\alpha_i)\varphi(\beta_{n-i})).
\]
That is, $\Endul_{kG}(\CE_V) \cong \prod_{i = 0}^\infty \Gamma_i$.
\end{prop}

\begin{rem} \label{rem:endoEM}
One of the features that makes the construction of $\CE_V$ and its 
endomorphism ring (as above) possible is that $kC$ acts trivially on 
the trivial module $k$. A similar thing can be done for any $kG$ module
$M$ that is inflated from a $kH$-module, so that $Z$ annihilates $M$. 
That is, we can construct $\CE_V(M) \cong \CE_V \otimes M$ from a minimal
projective $kH$-resolution of $M_{\downarrow H}$ with the action of 
$Z$ given by formulas as above. The final result 
is that $\Endul_{kG}(\CE_V(M))$ is, as a $k$-vector space,
the direct product $\prod_{i = 0}^\infty \Ext_{kH}^i(M,M)$.
The computation of the product poses similar problems as above, though
products of even degree elements coincide with the cohomology products. 
\end{rem}


\section{Finite generation} \label{sec:fingen}
Assume that $kG$ is a unipotent commutative group scheme. 
In this section, we show that if the restriction 
of a module $M$ along a $\pi$-point is 
a compact object in the stable category, then 
$\Homul_{kG}(\CE_V,M)$
is finitely generated as a module over the endomorphism ring of the 
unit object $\CE_V$ in the colocalized category. This leads to definition of a
local support variety for such a module. 
In the proof, we assume that the $\pi$-point is a map of Hopf algebras. 
This can be done as in Lemma \ref{lem:extendHopf}. However, in the 
end, the consequences of the analysis do not depend on the Hopf algebra 
structure. Hence the consequences hold without the assumption. 
Note that this is not an issue in the case that 
$kG$ is the restricted enveloping algebra
of a commutative Lie algebra with trivial $p^{th}$-power operation,
 as we saw in 
Example \ref{ex:restr-envel}. The proof also uses the detailed
knowledge of the structure of the idempotent module $\CE_V$.

We first note a general principle.

\begin{lemma} \label{lem:gen1}
Suppose that $V$ is a closed subvariety of $V_G(k)$. For any $kG$-module 
$M$, $\Homul_{kG}(\CE_V, M) \cong \Homul_{kG}(\CE_V, \CE_V \otimes M)$
as modules over $\Endul_{kG}(\CE_V)$.
\end{lemma}

\begin{proof}
This follows from the distinguished triangle for $M$ associated to the 
tensor ideal $\CM_V$:
\[
\xymatrix{
\dots \ar[r] & \Omega(\CF_V) \otimes M \ar[r] &
\CE_V \otimes M \ar[r] & k \otimes M \ar[r] & 
\CF_V \otimes M \ar[r] & \dots
}
\]
in that $\Homul_{kG}(\CE_V, \Omega(\CF_V) \otimes M) = 0 = 
\Homul_{kG}(\CE_V, \CF_V \otimes M)$.  
\end{proof}

Let $kG = k[t_1, \dots, t_r]/(t_1^p, \dots, t_r^p)$.
We assume the notation of Theorem \ref{thm:strEV}.
In particular, we let $\alpha:A = k[t]/(t^p) \to kG$ be a $\pi$-point 
that is a map of Hopf algebras. Let $kC$ be the image of $\alpha$
with $Z = \alpha(t)$. We choose a complementary subalgebra $kH$ so 
that $kG \cong kH \otimes kC$. Let $V$ be the variety consisting of the 
closed point represented by $\alpha$.

The module $\CE_V$ has a sequence
of finitely generated submodules 
\[
\CE_0 \subseteq \CE_1 \subseteq \CE_2 \subseteq \dots
\]
where 
\[
(\CE_n)_{\downarrow H} = P_0 \oplus (P_1 \otimes U) \oplus P_2 
\dots \oplus (P_n \otimes W)
\]
for $W=U$ if $n$ is odd and $W=k$ if $n$ is even, exactly
as in Theorem  \ref{thm:strEV}. Let $\iota_n: \CE_n \to \CE_V$ be the 
inclusion.  Suppose that  $M$ is a $kG$-module such that $\alpha^*(M)$ 
is compact in the stable category $\Stmod(A)$. 

\begin{defi} \label{def:degree}
For an element $\zeta \in \Homul_{kG}(\CE_V, M)$, the 
degree of $\zeta$ is the largest integer $d$ such that 
there exist a representative $\zeta^\prime \in \Hom_{kG}(\CE_V, M)$
of the class of $\zeta$ such that $\zeta^\prime\iota_{d-1}= 0$.
The leading term of $\zeta$ is the class in $\Homul_{kG}(\CE_d/\CE_{d-1}, M)$
of the map induced by $\zeta^\prime\iota_{d}$, where $\zeta^\prime$ 
is a representative of the degree-$d$ element $\zeta$ with 
$\zeta^\prime\iota_{d-1}= 0$.
\end{defi}

\begin{lemma} \label{lem:degree}
Suppose that $\zeta: \CE_V \to M$ is a $kG$-homomorphism such that, 
for some $n$, $\zeta(\CE_n) = \{ 0 \}$ and the induced map 
$\hat{\zeta}: \CE_{n+1}/\CE_n \to M$ factors through a $kG$-projective
module. Then there exist $\gamma:\CE_V \to M$ such that 
$(\zeta-\gamma)(\CE_{n+1}) = \{ 0 \}$ and $\gamma$ factors 
through a projective $kG$-module. In particular, the degree of 
$\zeta$ is at least $n+2$. 
\end{lemma}

\begin{proof}
First notice that since $\hat{\zeta}$ factors through a projective module, 
then so does $\zeta\iota_{n+1}$.
So, we have a diagram
\[
\xymatrix{
\CE_{n+1} \ar[rr]^{\zeta\iota_{n+1}} \ar[d]_{\iota_{n+1}} \ar[dr]^\beta && M \\
\CE_V \ar@{.>}[r]^{\mu} & P \ar[ur]^{\alpha}
}
\]
where $P$ is projective and $\alpha\beta = \zeta\iota_{n+1}$. The point is that
$P$ is also injective, so that there exists $\mu$ with $\mu\iota_{n+1} 
= \beta$. Then let $\gamma = \alpha\mu$. 
\end{proof}

Using the last lemma it is an easy exercise to prove the following. We leave
this exercise to the reader. 

\begin{prop} \label{prop:degree}
Every element $\zeta$ in $\Homul_{kG}(\CE_V, M)$ has a finite degree and 
well defined leading term. In particular, if $\zeta(\CE_{d-1}) = \{0\}$,
then the degree of $\zeta$ is $d$ if and only if $\zeta\iota_d$ does 
not factor through a projective module.
\end{prop}

\begin{prop} \label{prop:factorP}
Suppose that $\gamma: \CE_n/\CE_{n-1} \to M$ is a $kG$-homomorphism. 
Then $\gamma$ factors through a projective module if and only if either 
$n$ is odd and $\gamma(\CE_n/\CE_{n-1}) \subset ZM$ or $n$ is even and 
$\gamma(\CE_n/\CE_{n-1}) \subset Z^{p-1}M$. 
\end{prop} 

\begin{proof}
We prove the case for $n$ odd. The case for $n$ even is even easier. 
The point is that for $n$ odd, $\CE_n/\CE_{n-1} \cong P_n \otimes U$.
Because $P_n$ is a free $kH$-module, it is a direct sum of copies of 
$kH$ and $Z$ acts trivially on it. Thus, $\CE_n/\CE_{n-1}$ is a direct
sum of copies of $kH \otimes U$. If a map from $kH \otimes U$ factors
through a projective module, then it factors through the injective  hull  
$\psi: kH \otimes U \to kG$, where $\psi(1 \otimes 1) 
= Z$. Thus, if $\gamma$ factors through a projective, then its image
is in $ZM$.  On the other hand, if we have, for some summand
of $\CE_n/\CE_{n-1}$ that is isomorphic to $kH \otimes U$, that 
$\gamma(1\otimes 1) = Zm$, for $m \in M$, then we 
can define $\theta: kG \to M$ by $\theta(1) = m$, and $\theta\psi$
coincides with $\gamma$ on that summand.
\end{proof}

We recall that  $\Endul_{kG}(\CE_V) \cong \prod_{n \geq 0} \HHH^*(H,k)$.
The action of a  nonzero element 
$\gamma \in \HHH^n(G,k) \subseteq \Endul_{kG}(\CE_V)$,
on $\CE_V$ is determined by the chain map $\{\gamma_* \}:P_* \to P_*$ of 
degree $n$. Hence, its degree as an element of $\Endul_{kG}(\CE_V)$
is $n$ and its leading term is the map 
$\nu: \ \CE_n/\CE_{n-1} \to  P_0 \cong \CE_0  \subseteq  \CE_V,$
defined by $\gamma_0$.
(See the discussion following Proposition \ref{prop:endoE1}.)
 
For any $i \geq 0$, there is a shift $\Omega^i(\nu)$ which is the 
induced map 
\[
\xymatrix{
\Omega^i(\nu) = \ \CE_{n+i}/\CE_{n+i-1}  \ar[r] & 
\CE_i/\CE_{i-1}.
}
\]
Note that, $\gamma(\CE_{n+i-1}) \subseteq \CE_{i-1}$, so there is such a map.

\begin{prop} \label{prop:mult1}
Suppose that an element $\zeta \in \Homul_{kG}(\CE_V,M)$ has degree $d$ 
and leading term $\mu$. Let $\gamma \in \Homul_{kG}(\CE_V, \CE_V)$ be 
an element with degree $n$ and leading term $\nu$. 
If $\mu\Omega^d(\nu) \neq 0$ (meaning a representative does not factor
through a projective module), then $\gamma\zeta$ has degree $d+n$
and $\mu\Omega^d(\nu)$ is its leading term. 
Otherwise, it has degree greater than $d+n$. 
\end{prop}

\begin{proof}
We have that $\gamma$ has a representative (call it $\gamma$) in 
$\Hom_{kG}(\CE_V, \CE_V)$ such that $\gamma(\CE_{n-1}) = 0$. Likewise, 
$\zeta$ has a representative such that $\zeta(\CE_{d-1}) = 0$. As a 
result, $\gamma\zeta(\CE_{n+d-1}) = 0$ and the composition $\gamma\zeta$
induces a map that is a composition
\[
\xymatrix{
\CE_{n+d}/\CE_{n+d-1} \ar[r] & \CE_{n}/\CE_{n-1} \ar[r] & M
}
\]
that is easily seen to be $\mu\Omega^d(\nu)$.
\end{proof}

To explain the products of leading elements, we should note the following.
For notation, let $\Gamma(M) = M^C/Z^{p-1}M$, where $M^C$ is the submodule
of $kC$-fixed points. With the assumption that $\alpha^*(M)$ is finite 
dimensional, so also is $\Gamma(M)$. 

\begin{lemma} \label{lem:hom1}
For any $n$ we have that 
\[
\Homul_{kG}(\CE_n/\CE_{n-1}, M) \cong \begin{cases} 
\Gamma (M) \otimes \HHH^n(H,k) & \text{ if } n \text{ is even },\\
\Gamma(M \otimes U) \otimes \HHH^n(H,k) & \text{ if } n \text{ is odd }.
\end{cases}
\]
\end{lemma}

\begin{proof}
We note that, if $n$ is even, then $Z$ annihilates $\CE_n/\CE_{n-1}$. 
Consequently, the image of
any homomorphism $\vartheta: \CE_n/\CE_{n-1} \to M$ lies in $M^C$. On the 
other hand, if this image lies in $Z^{p-1}M$, then $\vartheta$ factors 
through a projective module. This proves the first case. For the case 
that $n$ is odd, we notice that 
\[
\Homul_{kG}(\CE_n/\CE_{n-1}, M) = \Homul_{kG}(P_n \otimes U, M) \cong
\Homul_{kG}(P_n, M \otimes U)
\]
since $U$ is self-dual. 
\end{proof}

Putting the last two results together we easily get the following. 

\begin{lemma} \label{lem:hom2}
In the notation of Proposition \ref{prop:mult1}, if $\zeta$ has
even degree $n$ and 
leading term $\mu = \sum_{i =1}^t m_i \otimes \vartheta_i$ 
for $m_i \in \Gamma(M)$ and $\vartheta_i \in \HHH^n(H,k)$
as in Lemma \ref{lem:hom1}, then $\mu\Omega^d(\nu) = 
\sum_{i =1}^t m_i \otimes \vartheta_i\nu$ where by $\vartheta_i\nu$ 
we mean the product in $\HHH^*(H,k)$.
\end{lemma}

We can now prove the main theorem of the section. 

\begin{thm} \label{thm:fingen}
Let $kG = k[t_1, \dots, t_r]/(t_1^p, \dots, t_r^p)$.
We let $\alpha:A = k[t]/(t^p) \to kG$ be a $\pi$-point
defined over $k$. 
Let $V$ be the variety consisting of the
closed point represented by $\alpha$.
If $M$ is a $kG$-module whose support variety consists of the
single point $V$, and if the restriction $\alpha^*(M)$ is a compact object in 
$\Stmod(A)$, then $\Homul_{kG}(\CE_V, M)$ is finitely generated
as a module over $\Endul_{kG}(\CE_V)$.
\end{thm}

\begin{proof}
As before, we write $kG \cong kH \otimes kC$ where $kC$ is the 
image of $\alpha$ and $kH$ is a complementary subalgebra. 
By Lemma \ref{lem:extendHopf}, there is a Hopf algebra structure on 
$kG$ such that $\alpha$ is a map of Hopf algebras. We assume this 
structure so that the previous results of this section hold. 
Let $\HHH_{ev}^* = \sum_{n \geq 0} \HHH^{2n}(H,k)$. Notice that 
$\HHH^*(H,k)$ is finitely generated over $\HHH_{ev}^*$. 

Let $\CL_n$ be
the collection of all leading terms of elements in $\Homul_{kG}(\CE_V, M)$
having degree at most $n$. Thus, 
\[
\CL_n \subseteq \sum_{j=0}^n (\Gamma(M) \oplus \Gamma(M \otimes U))
\otimes \HHH^j(H,k). 
\]
Let $\CH = \sum_{j\geq 0}  (\Gamma(M) \oplus \Gamma(M \otimes U))
 \otimes \HHH^j(H,k)$. Let $I_n \subset \CH$ be the $\HHH_{ev}^*$-submodule
of $\CH$ generated by $\CL_n$. Then we have an increasing sequence of
submodules $I_0 \subseteq I_1 \subseteq I_2 \dots $. Because, $\CH$ is 
finitely generated over $\HHH_{ev}^*$, which is noetherian, 
the sequence terminates. That is,  $I = \cup_{n} I_n$ is a finitely 
generated $\HHH_{ev}^*$-submodule of $\CH$  and $I = I_N$ for some $N$.  

For each $j = 0, \dots, N$, choose a finite set of elements 
$U_j \subseteq \Homul_{kG}(\CE_V, M)$ such that the set of leading terms of 
the elements of $U_j$ generated $I_j/I_{j-1}$. 
Let $U = \{u_1, u_2, \dots, u_t\} = \cup_{j=0}^N U_j$. We claim 
that $U$ is a set of generators for $\Homul_{kG}(\CE_V, M)$ as a module 
over $\HHH_{ev}^*$. 

For each $j$, let $v_j$ be the leading term of $u_j$. Suppose that 
$\gamma \in \Homul_{kG}(\CE_V, M)$.  Let $s$ be the 
degree of $\gamma$ and let $\mu$ be its leading term. 
Then we can write $\mu = \sum_{i = 1}^t u_i\zeta_i$ where each $\zeta_i$
is an element of $\HHH^*(H,k)$. Let $\zeta_i$ denote also an element 
in $\Endul_{kG}(\CE_V) = \prod_{i} \HHH^i(H,k)$, that has leading term 
$\zeta_i$. Let $\beta_1 = \sum_{i=0}^t v_i\zeta_i$. Note that $\beta_1$ 
has degree $s$ and 
has the same leading term as $\gamma_1 = \gamma$. Then 
$\gamma_2 = \gamma_1-\beta_1$ has larger degree than $s$, the degree of 
$\gamma$.  Now repeat this process with $\gamma_2$ in place of $\gamma$.
We get $\beta_2$ with the same degree and leading term as $\gamma_2$, 
so that $\gamma_3 = \gamma_2- \beta_2$ has higher degree than $\gamma_2$.

Taking a limit of this process we obtain a sequence $\beta_1, \beta_2, \dots$
such that $\gamma = \beta_1 +\beta_2 + \dots$. Note that we can add the 
elements of this infinite sequence because they have different degrees. 
For each $j$, we have that  
$\beta_j = \sum_{i=1}^t v_i \zeta_{i,j}$ for some elements
$\zeta_{i,j} \in \Endul_{kG}(\CE_V)$. Thus,
\[
\alpha = \sum_{j \geq 1} \beta_j = 
\sum_{i= 0}^t v_i (\sum_{j \geq 1} \zeta_{i,j})
\]
It follows that the $\{v_i\}_{1 \leq i \leq t}$ is a set of generators
for $\Homul_{kG}(\CE_V, M)$ as asserted. 

Finally, we point out that the conclusion of the theorem is independent of 
the Hopf algebra structure on $kG$. Consequently, the theorem is true even 
without the assumption on the Hopf 
nature of the $\pi$-point at the beginning of the proof. 
\end{proof}

\begin{rem} \label{rem:dualizing}
In \cite{BIKP2} it is proved that Theorem \ref{thm:fingen} has a strong 
converse. That is, the authors of \cite{BIKP2} show that the condition 
on the module $M$ in the hypothesis is equivalent to the condition that
the module by dualizable, which in turn is equivalent to the finite
generation condition on $\Homul(\CE_V, M)$
\end{rem}


\section{(Co)Local supports} \label{sec:locsupp}
If a module is finitely generated over its base ring, then its annihilator
in that ring is an ideal, and we can define the support variety of 
the module. Rings constructed as $\Endul_{kG}(\CE_V)$, as in 
Section \ref{sec:endomorph}, are not often finitely generated as
algebras over the base field $k$. We expect that the prime ideal spectrum
of such a ring is chaotic. 
None the less, the annihilator of a 
module such as $\Homul_{kG}(\CE_V, M)$ for a $M$ a $kG$-module, is an 
invariant and can be used to distinguish modules. In this section we take
a brief look at some of the possibilities. In particular, we prove a form
of a realization theorem. 

To this end, let $\FA_V(M) = \Ann_{\CC}(\Homul_{kG}(\CE_V,M))$ where 
$\CC = \Endul_{kG}(\CE_V)$ and $V$ is a closed subvariety of $V_G(k)$.
From Lemma \ref{lem:gen1}, we see that if $M$ is a finite dimensional 
$kG$-module, then $\FA_V(M) = \FA_V(M \otimes \CE_V)$. 

We begin with a couple of easy result. 

\begin{prop} \label{prop:loctri}
Suppose that the variety of $M$ is in $V$ for $V$ as above, a close
subvariety of $V_G(k)$. Then, 
$\Homul_{kG}(\CE_V,M) \cong \Homul_{kG}(k,M)$. 
\end{prop}

\begin{proof}
In the stable category $M \otimes \CF_V = 0$. Hence,
$\Homul_{kG}(\CE_V,M \otimes \CF_V) = 0$. Thus the result follows from
the distinguished triangle $\CS_V$ involving $\CE_V$.
\end{proof}

The next result follows directly from Remark \ref{rem:endoEM}.

\begin{prop} \label{prop:inflate}
Suppose that $kG \cong kH \otimes kC$ as in Theorem \ref{thm:strEV}. That is,
$kC \cong k[t]/(t^p)$ is the image of a $\pi$-point $\alpha$ whose class $V$ in 
$V_G(k)$ is a closed point, and $kH$
is the group algebra of an elementary abelian $p$-group. Suppose that $M$ 
is the inflation of a finitely generated $kH$-module. Then, $\FA_V(M)$ 
consists of all tuples $(\gamma_1, \gamma_2, \dots)$ such that 
$\gamma_i$ annilates $\HHH^*(G, M)$ for all $i$.  
\end{prop}

Before further exploration, we need a technical result. 

\begin{lemma} \label{lem:ann-tri}
Let $V$ be a closed subvariety of $V_G(k)$. Suppose that 
\[
\xymatrix{
{} \ar[r] & L \ar[r]^\alpha & M \ar[r]^\beta & N \ar[r]  & \Omega^{-1}(L) 
}
\]
is a triangle of $kG$-modules and that $\zeta \in \FA_V(N)$, $\gamma \in 
\FA_V(L)$. Then $\zeta\gamma \in \FA_V(M)$.
\end{lemma}

\begin{proof}
For $\mu \in \Homul_{kG}(\CE_V, M)$, we have that 
$\beta(\mu\zeta) = (\beta\mu)\zeta = 0$ in the stable category. 
Consequently, there exists $\nu \in \Homul_{kG}(\CE_V, L)$ such
that $\mu\zeta = \alpha\nu$. Thus, $\mu\zeta\gamma = \alpha(\nu\gamma) = 0$.
\end{proof}

Let $\ell = \ell(kG)$ denote the Loewy length of $kG$, {\it i. e.} the least 
integer $n$ such that $\Rad^n(kG) = \{0\}$. Recall that for any 
$kG$-module $M$, $\Rad^{\ell}(M) = \Rad^{\ell}(kG) M = \{0\}$.

\begin{prop}  \label{prop:tensor-ann}
Let $V$ be a closed subvariety of $V_G(k)$ and $M$ a $kG$-module. 
Suppose that $\zeta \in \FA_V(M)$.
Then, for any finite dimensional $kG$-module $N$, 
$\zeta^{\ell(kG)} \in \FA_V(M \otimes N)$.
\end{prop}

\begin{proof}
First notice that it is sufficient to prove the proposition in the case
that $G$ is a $p$-group. This is because whenever a $kG$-map has the property 
that its restriction to the Sylow $p$-subgroup $P$ of $G$ factors through a 
$kP$-projective module, then the map factors through a $kG$-projective module.

Let $R_j = \Rad^j(N)$, so that we have a radical filtration 
\[
\{0\} = R_n \subset R_{n-1} \subset \dots \subset R_1 \subset R_0 = N,  
\qquad \qquad n \leq \ell,
\]
where every one of the quotients $R_i/R_{i+1}$ is a direct sum of copies 
of the trivial module $k$. Let $M_i = M \otimes R_i$. 
Thus, for $0 \leq i < n$ we have a triangle
\[
\xymatrix{
{}\ar[r] & \Homul_{kG}(\CE_V, M_{i+1}) \ar[r] & 
\Homul_{kG}(\CE_V, M_i) \ar[r] & \Homul_{kG}(\CE_V, M_i/M_{i+1}) \ar[r] & {}.
}
\] 
Fix the index $i$. The quotient  
\[
M_i/M_{i+1} \cong M \otimes (R_i/R_{i+1}) 
\cong \oplus_{j \in J} M \otimes k \cong \oplus_{j \in J} M.
\]
for some finite indexing set $J$.

Let $\hat{\zeta}: \CE_V \to \CE_V$ be a representative of the class of 
$\zeta$. Choose any homomorphism $\mu: \CE_V \to M \otimes (R_i/R_{i+1})$.
The composition $\mu\hat{\zeta}$ factors through
a projective module since $R_i/R_{i+1}$ has finite dimension.  
Thus, we have shown that $\zeta$ annihilates 
$\Homul_{kG}(\CE_V, (M \otimes R_i)/(M \otimes R_{i+1}))$.

Using this fact we prove the proposition by applying  Lemma \ref{lem:ann-tri}
repeatedly, beginning with $i = n-1$.
\end{proof}

There is something of a realization theorem. First we need the 
following. 

\begin{lemma} \label{lem:commut}
Suppose that $V$ is a closed subvariety of $V_G(k)$. 
Suppose that $\zeta \in \Endul_{kG}(\CE_V)$ and that 
$\gamma \in \Homul_{kG}(\CE_V, \Omega(\CE_V))$. Then 
$\gamma\zeta = \Omega(\zeta)\gamma$ in $\Stmodg$.
\end{lemma}

\begin{proof}
The lemma is a consequence of general principles. That is, the 
$kG$-module $\CE_V$ is the unit object in the subcategory $\CM_V$ and 
its graded endomorphism ring $\End_{\CM}^*(\CE_V) = 
\sum_{n \in \bZ} \Homul_{kG}(\CE_V, \Omega^{-n}(\CE_V))$ is graded 
commutative. 
\end{proof}

\begin{thm} \label{thm:realize}
Suppose that $V$ is a closed subvariety of $V_G(k)$.  
Choose an element $\zeta \in \Endul_{kG}(\CE_V)$ and let $M_\zeta$ be 
the third object in the triangle of $\zeta$:
\[
\xymatrix{
{} \ar[r] & \Omega(\CE_V) \ar[r]^\gamma & 
M_\zeta \ar[r]^\beta & \CE_V \ar[r]^\zeta & \CE_V \ar[r] & 
\Omega^{-1}(M_\zeta) \ar[r] & {}
}
\]
Then the radical of the 
ideal $\FA_V(M_{\zeta})$ is the same as that of the ideal generated by
$\zeta$.
\end{thm} 

\begin{proof} We use the diagram below to show that 
if $\mu \in \FA_V(M_{\zeta})$, then
$\mu^\ell$ is a multiple of $\zeta$. 
To start, notice that  
$\Omega^{-1}(M_\zeta) \cong \Omega^{-1}(k) \otimes M_\zeta$
in the stable category and by Proposition \ref{prop:tensor-ann}, 
$\mu^{\ell}$ is in $\FA_V(\Omega^{-1}(M_\zeta)).$
In the diagram the vertical maps are right multiplication by $\mu^\ell$. 
\[
\xymatrix{
{} \ar[r] & \Homul_{kG}(\CE_V, \CE_V) \ar[r]^\zeta \ar[d]^{\mu^\ell_*} &
\Homul_{kG}(\CE_V, \CE_V) \ar[r]^{\beta \quad}  
\ar[d]^{\mu^\ell_*} \ar@{.>}[dl]_\sigma  &
\Homul_{kG}(\CE_V, \Omega^{-1}(M_\zeta)) \ar[r] \ar[d]^{\mu^\ell_*} & {} \\
{} \ar[r] & \Homul_{kG}(\CE_V, \CE_V) \ar[r]^\zeta &
\Homul_{kG}(\CE_V, \CE_V) \ar[r]^{\beta \quad}  &
\Homul_{kG}(\CE_V, \Omega^{-1}(M_\zeta)) \ar[r]  & {}
}
\]
Because the composition $\beta\mu^\ell_* = \mu^\ell_*\beta$ is zero
on the middle term $\Homul_{kG}(\CE_V, \CE_V)$, we have that there is a 
map $\sigma: \Homul_{kG}(\CE_V, \CE_V) \to \Homul_{kG}(\CE_V, \CE_V)$
such that $\zeta\sigma = \mu^\ell_*.$ Thus we have that 
$\sigma(\Id_M)\circ \zeta = \mu^\ell_*(\Id_M) = 
\Id_M \circ \mu^\ell = \mu^\ell$. 

Next we show that $\zeta^2 \in \FA_V(M_\zeta)$.  
First, we see from the exact sequence 
\[
\xymatrix{
{} \ar[r] & \Homul_{kG}(\CE_V, M_\zeta) \ar[r]^{\beta_*} &
\Homul_{kG}(\CE_V, \CE_V) \ar[r]^\zeta  &
\Homul_{kG}(\CE_V, \CE_V) \ar[r] & {} \\
}
\]
that $\zeta\beta = 0$ where here $\zeta$ means left composition with
$\zeta$. Hence, in the diagram 
\[
\xymatrix{
{} \ar[r] & \Homul_{kG}(\CE_V, \Omega(\CE_V)) 
\ar[r]^{\gamma_*} \ar[d]^{\zeta^*} &
\Homul_{kG}(\CE_V, M_\zeta) \ar[r]^{\beta_*}  
\ar[d]^{\zeta^*} \ar@{.>}[dl]_\tau  & {}             
\Homul_{kG}(\CE_V, \CE_V) \ar[r] \ar[d]^{\zeta^*} & {} \\
{} \ar[r] & \Homul_{kG}(\CE_V, \Omega(\CE_V)) \ar[r]^{\gamma_*} &
\Homul_{kG}(\CE_V, M_\zeta) \ar[r]^{\beta_*}  &
\Homul_{kG}(\CE_V, \CE_V) \ar[r]  & {}
}
\]
the composition $\beta\zeta$ is also zero. Note here, the down arrows 
marked $\zeta^*$ denote right multiplication by $\zeta$. However, there is
no problem since $\Homul_{kG}(\CE_V, \CE_V)$ is commutative. 
As a result, the map $\tau$ exist with 
$\gamma\tau = \zeta$. 

Finally, we consider the diagram 
\[
\xymatrix{
{} \ar[r] & \Homul_{kG}(\CE_V, \Omega(\CE_V)) 
         \ar[r]^{\Omega(\zeta)_*} \ar[d]^{\zeta^*} &
\Homul_{kG}(\CE_V, \Omega(\CE_V)) \ar[r]^{\gamma_*}  \ar[d]^{\zeta^*} &
\Homul_{kG}(\CE_V, M_\zeta) \ar[r] \ar[d]^{\zeta^*} & {} \\
{} \ar[r] & \Homul_{kG}(\CE_V, \Omega(\CE_V)) \ar[r]^{\Omega(\zeta)_*} &
\Homul_{kG}(\CE_V,  \Omega(\CE_V)) \ar[r]^{\gamma_*}  &
\Homul_{kG}(\CE_V, M_\zeta) \ar[r]  & {}.
}
\]
For $\mu \in \Homul_{kG}(\CE_V, M_\zeta)$ we have that $\mu\zeta = 
\gamma \circ (\tau(\mu))$.  Here we regard $\tau(\mu)$ as an element of 
the object in the middle of the upper row of the diagram. Recall that 
$\gamma_*$ is left composition with the map 
$\gamma: \Omega(\CE_V) \to M_\zeta$
in the triangle that contains $\zeta$.
Thus we have that 
\[
\mu \zeta^2 = 
(\gamma(\tau(\mu))\zeta = (\gamma \circ \tau(\mu)) \circ \zeta
= \gamma \circ (\Omega(\zeta) \circ \tau(\mu)) =
 \gamma(\Omega(\zeta)(\tau(\mu))) = 0,
\]
by Lemma \ref{lem:commut} and the triangle of $\gamma$. 
\end{proof}

Suppose that $0 \neq \zeta \in \HHH^n(G,k)$. Assume that $n$ is even if $p$ is
odd. Associated to $\zeta$ is a module that generates the thick tensor
ideal in $\stmodg$ of all modules whose support variety is contained in 
the closed set of all prime ideals that contain $\zeta$. The module $L_\zeta$
is defined to be the kernel the map $\hat{\zeta}$
\[
\xymatrix{
0 \ar[r] & L_\zeta \ar[r] & \Omega^n(k) \ar[r]^{ \ \hat{\zeta}} & k \ar[r] & 0
}
\]
where $\hat{\zeta}$ is a cocyle representing $\zeta$.
Suppose that $V$ is a single point as in the last theorem. Tensoring 
the above sequence with $\CE_V$ yields an exact sequence 
\[
\xymatrix{
0 \ar[r] & L_\zeta \otimes \CE_V \ar[r] & 
\Omega^n(k) \otimes \CE_V \ar[r]^{\qquad 1 \otimes \hat{\zeta}} & 
\CE_V \ar[r] & 0.
}
\]

But recall that $\CE_V$ is periodic of period 1 if $p=2$ and period 2 
otherwise. That is, for example, if $p > 2$, then the tensor of $\CE_V$
with the sequence 
\[
\xymatrix{ 
0 \ar[r] & k \ar[r] & k(G/H) \ar[r]^{Z} & k(G/H) \ar[r] & k \ar[r] & 0 
}
\]
where $k(G/H)$ is the permutation module 
on which $kH$-acts trivially and the middle 
map is multiplication by $Z$, yields a sequence with $\CE_V$ at the ends
and middle terms projective. Thus we have that $\Omega^n(k) \otimes \CE_V 
\cong \CE_V$ in the stable category. The consequence of this is that 
$L_\zeta \otimes \CE_V$ is isomorphic to some $M_\gamma$. So what is $\gamma$?
We sketch a proof of an answer to the question. 

For notation, we recall that $\HHH^*(C,k) \cong k[\mu]$ 
if $p=2$ and, otherwise, 
$\HHH^*(C,k) \cong k[\mu, \nu]/(\mu^2)$ where $\mu$ and $\nu$ are in 
degrees one and two. Thus for all $n$, $\HHH^n(C,k)$ has dimension one and 
has a basis element which we demote $\mu_n$. So in the case $p >2$, 
$\mu_{2n+1} = \mu\nu^n$ while $\mu_{2n} = \nu^n$. Recall also that 
$\HHH^*(G,k) = \HHH^*(H,k) \otimes \HHH^*(C,k)$. 

\begin{prop} \label{prop:zeta-gamma}
Assume the hypothesis of the Theorem \ref{thm:realize}. Let $\zeta$ be 
an element in $\HHH^n(G,k)$. Write $\zeta = \sum_{i=0}^n \mu_i \gamma_{n-i}$
for $\gamma_{n-i} \in \HHH^{n-i}(H,k)$. Then in the stable category, 
$L_\zeta \otimes \CE_V \cong M_\gamma$ where 
\[
\gamma = (\gamma_0, \dots, \gamma_n, 0,0,0,0,\dots ) \in 
\prod_{i\geq 0} \HHH^i(H,k) \cong \Endul_{kG}(\CE_V).
\]
\end{prop}

\begin{proof}
The secret to the proof is finding a splitting of the middle vertical map
in the diagram:
\[
\xymatrix{
0 \ar[r] & L_\zeta \otimes \CE_V \ar[r] \ar[d]&
\Omega^n(k) \otimes \CE_V \ar[r]^{\qquad 1 \otimes \hat{\zeta}} \ar[d] & 
\CE_V \ar[r] \ar@{=}[d]  & 0 \\
0 \ar[r] & M_\zeta \ar[r] &
\CE_V \ar[r]  & \CE_V \ar[r] & 0.
}
\]
The map itself can be taken to be $(\mu_n \otimes 1) \otimes 1$, since
the kernel $L_{\mu_n}$ of the cocycle $\mu_n \otimes 1: \Omega^n(k) \to k$ is 
free on restriction to $kC$, and $L_{\mu_n} \otimes \CE_V$ is a projective
module. 

We construct the splitting $\theta$ as follows. Let $k\hat{G} = 
kC \otimes kH \otimes kH$ and conctruct the idempotent $\CE$-module $\hat{\CE}$
as in Theorem \ref{thm:strEV}. If $p = 2$, then on restriction 
to $kH \otimes kH$, $\hat{\CE}$
is a sum of the terms of the projective resolution $((P \otimes P)_*, 
\varepsilon \otimes \varepsilon)$ where $(P_*, \varepsilon)$ is the 
minimal projective $kH$-resolution of $k$. 
For $p$ odd, the restriction is a sum of terms of the form 
$(P \otimes P)_n \otimes X$ where $X$ is either $U$ or $k$ depending 
on the parity of $n$. In either case we have a chain map 
$\hat{\psi}: P_* \to (P \otimes P)_*$ that defines cup product on 
cohomology and splits both $1 \otimes \varepsilon$ and $\varepsilon \otimes 1$.
This gives us a map $\psi: \CE_V \to \hat{\CE}$. At the same
time, $\hat{\CE}$ when viewed as a $kG$-module by restriction (taking the 
diagonal embedding of $kH$ into $kH \otimes kH$) is very naturally isomorphic
to $\CE_V \otimes \CE_V$.

Next we notice that the module $V$ whose restriction to $kH$ is the direct sum
\[
V_{\downarrow kH} = P_0 \oplus (P_1 \otimes U) \oplus P_2 \oplus
(P_3 \otimes U) \oplus \dots \oplus (P_{n-1} \otimes U) \oplus 
P_n/\partial(P_{n+1})
\]
where $U =k$, if $p=2$. Here, $Z$ acts as in the formula in Theorem
\ref{thm:strEV}, where defined, and by the induce map $Z: P_n \to P_{n+1}$
on the last summand. But now, we see from the formula for minimal 
projective $kG$-resolution of $k$ as the tensor product of $kC$ and $kH$ 
resolutions, that $V \cong \Omega^n(k)$. Letting $\sigma: \CE_V \to V$
be the obvious quotient, we get the splitting 
\[
\xymatrix{ 
\CE_V \ar[r]^{\psi \qquad} & \hat{\CE} \cong \CE_V \otimes \CE_V 
\ar[r]^{ \ \sigma \otimes 1} & \Omega^n(k) \otimes \CE_V .
}
\]
Now the theorem follows by composition of maps, using the splittings. 
\end{proof}

In the case that $\gamma_0 \ne 0$, then the $\gamma$ is a 
unit in $\Endul_{kG}(\CE_V)$, implying that $L_\zeta \otimes \CE_V$ is
the zero module. The situation is consistent with what we know
in that $\gamma_0 \ne 0$ implies that $\res_{G, C}(\zeta) \neq 0$, and 
hence $L_{\zeta}$ is projective on restriction to $kC$ and 
$L_{\zeta} \otimes \CE_V$ is a projective module. 

We end by pointing out another curosity. 
Again, suppose that $G$, $H$, $C$, $V$ are as in Proposition \ref{prop:inflate}.
Let $\zeta$ be an element $\Endul_{kG}(\CE_V)$ with $\zeta$ not invertible, 
Let $\CN_\zeta$ be the subcategory of all $kG$-modules $M$ 
whose variety is in $V$ and such that such that $\zeta$ annihilates
$\Homul_{kG}(\CE_V, M)$. Then $\CN_\zeta$ is a thick subcategory of 
$\Stmodg$. It is not however a localizing subcategory. We know this because
the localizing tensor ideal  have been classified \cite{BIK, BIKP}.
But also, we can see by Theorem \ref{thm:realize}, that the direct sum 
$M = \sum_{n=1}^\infty M_{\zeta^n}$ is not in $\CN_\zeta$. 
Moreover, under some mild assumptions, the module $\CE_V$ can be recovered
as the third object in a triangle involving an endomorphism of 
the module $M$, above. This is a sort of homotopy colimit construction
and will be explored in subsequent work. 


\end{document}